\documentclass[11pt]{article}
\usepackage{amscd, amsthm, color}
\usepackage{mathrsfs}
\input epsf
\usepackage{amssymb}
\usepackage{amsmath}
\usepackage{amsfonts}
\input amssym.def
\newtheorem{theorem}{Theorem}[section]

\newtheorem{lemma}[theorem]{Lemma}
\newtheorem{prop}[theorem]{Proposition}

\newtheorem{re}[theorem]{Remark}
\newtheorem{no}[theorem]{Notation}
\newtheorem{definition}[theorem]{Definition}
\theoremstyle{definition}

\definecolor{wco}{rgb}{0.5,0.2,0.3}

\numberwithin{equation}{section}

\begin{document}
\def\beg{\begin}
\def\beq{\begin{equation}}
\def\enq{\end{equation}}
\title{ Gelfand and Kolmogorov numbers of Sobolev embeddings\\
of weighted function spaces II
}
\author{Shun Zhang $^{a,\,b}$,\ \ \  Gensun Fang
$^{b,\,}$\footnote{Corresponding author.
\newline\indent\ \,
E-mail addresses:  shzhang27@163.com (S. Zhang), fanggs@bnu.edu.cn (G. Fang), flhuang@ahu.edu.cn
(F. Huang).}\ ,\ \  Fanglun Huang
$^{c}$
\\ { $^{\rm a}$ {\small School of Computer Science and Technology, Anhui
University,
 Hefei 230601, Anhui, China}}
\\ {$^{\rm b}$ {\small School of Mathematical Sciences, Beijing Normal University,
Beijing 100875, China}}
 \\ {$^{\rm c}$ {\small School of Mathematical Sciences, Anhui
University,
 Hefei 230601, Anhui, China}}}
 \maketitle
\begin{abstract}
We consider the Gelfand and Kolmogorov numbers of compact embeddings
between weighted function spaces of Besov and Triebel-Lizorkin type
with polynomial weights in the non-limiting case. Our main purpose
here is to complement our previous results in \cite{ZF10} in the
context of the quasi-Banach setting, $0 < p, q \le \infty$. In
addition, sharp estimates for their approximation numbers in several
cases left open in Skrzypczak (2005) \cite{Sk05} are provided.
 \end{abstract}

\section{Introduction}
Since the 1950s, under the influence of the work Kolmogorov
\cite{Kol36}, a new perspective on approximation theory has
developed. The study of widths, optimal recovery and computational
complexity has received much attention, see
\cite{BBS10,DNS10,FHL05,MR85,MT09,NW10,OW10,Pin85,TWW88} for a
survey. In particular, one of the major tasks is to determine the
exact (or asymptotic exact) degree of various n-widths of some
classical classes of functions in different computational settings,
and find optimal algorithms.

The main aim of this paper is to complement our previous results
obtained in \cite{ZF10} on the Sobolev embeddings between weighted
function spaces of Besov and Triebel-Lizorkin type. There we
considered the case where the ratio of the weights $w(x)$ is of polynomial type, and established the asymptotic
order of the Gelfand and Kolmogorov numbers of the corresponding
embeddings where the spaces involved were Banach spaces, $1\le
p,q\le\infty$. The main reason for these restrictions was that there
were still several gaps left open, especially on the estimates for
Gelfand numbers of the Euclidean ball in the quasi-Banach case.
Fortunately, we recently became aware of some surprising results
from Foucart et al. \cite{FPRU} and Vyb\'iral \cite{Vy08}.

In the present paper, we extend completely the results of
\cite{Sk05,ZF10} to the quasi-Banach space case $0<p, q\le \infty$
in the so-called non-limiting situation. In particular, we provide
sharp estimates for the asymptotic behavior of the Gelfand and
Kolmogorov numbers  in several cases of the Banach space setting
left open in \cite{ZF10}. Also, several gaps for the approximation
numbers left open in \cite{Sk05} are closed.

The discretization technique adopted in \cite{ZF10} is still
effective in the quasi-Banach setting. The characterization of
weighted Besov spaces in terms of wavelets was proved by Haroske and
Triebel \cite{HT05} in the quasi-Banach case. Moreover, the operator
ideal technique works also in this case.
 Historically, Pietsch \cite{Pie78,Pie87} developed
the theory of operator ideals and s-numbers. The technique of
estimating single s-numbers or entropy numbers via estimates of
ideal (quasi-)norms derives from ideas of Carl \cite{Car81}. In the
1980s this technique was frequently used in operator theory, in
eigenvalue problems for Banach space operators, etc. However, the
operator ideal technique remained unknown in the function spaces
community for many years. As far as we know, it was applied for the
first time in \cite{Ku03,KLSS03a}, which both appeared in 2003.

In a sequel to \cite{ZF10}, we concentrate on the Sobolev
embeddings,
\begin{equation}\label{BB}
B_{p_1,q_1}^{s_1}(\mathbb{R}^d, w_\alpha)\hookrightarrow
B_{p_2,q_2}^{s_2}(\mathbb{R}^d),
\end{equation}
 with polynomial weights
\begin{equation}\label{w_a}
w_\alpha(x):=(1+|x|^2)^{\alpha/2}
\end{equation}
for some exponent $\alpha>0.$ The non-limiting case means that
$\delta\neq\alpha$.

The organization of this paper is as follows. In Section 2, we
recall some definitions and related properties, and present our main
results. In Section 3, we collect several necessary estimates of
Kolmogorov and Gelfand numbers of the Euclidean ball. Main proofs
are shifted to Section 4. Finally, in Section 5 we complement the
known results of Skrzypczak \cite{Sk05} for the approximation
numbers, and compare these three quantities of the function space
embeddings. Our main assertions are Theorem \ref{T1} and Theorem
\ref{T2}, which generalize the main results in \cite{ZF10}.

Let us make an agreement throughout this paper,
\begin{equation}\label{emB_con} -\infty<s_2<s_1<\infty,\
0< p_1, p_2, q_1, q_2\leq\infty\ \ {\rm and}\ \
\delta=s_1-s_2-d(\frac 1{p_1}-\frac 1{p_2})>0
\end{equation}
if no further restrictions are stated.

\begin{no}
By the symbol ` $\hookrightarrow$'  we denote continuous embeddings.

By $\mathbb{N}$ we denote the set of natural numbers, by\
$\mathbb{N}_0$\ the set\, $\mathbb{N}\cup\{0\}$.

Identity operators will always be denoted by {\rm id}. Sometimes we
do not indicate the spaces where {\rm id} is considered, and
likewise for other operators.

Let $X$ and $Y$ be complex quasi-Banach spaces and denote by
$\mathcal {L}(X, Y)$ the class of all linear continuous operators
$T:\,X \rightarrow\, Y.$ If no ambiguity arises, we write $\|T\|$
instead of the more exact versions $\|T ~|~ \mathcal {L}(X, Y)\|$ or
$\|T:X\rightarrow Y\|$.

The symbol $a_n \preceq b_n$ means that there exists a constant $c
> 0$\ such
that\ $a_n\le c\,b_n$\ for all\ $n\in\mathbb{N}.$\ And $a_n \succeq
b_n$ stands for $b_n \preceq a_n,$\ while $a_n\sim b_n$ denotes\
$a_n\preceq b_n \preceq a_n.$

All unimportant constants will be denoted by $c$ or $C$, sometimes
with additional indices.
\end{no}

\section{Main results}

First let us recall the definitions of Kolmogorov and Gelfand
numbers, cf. \cite{Pin85}. We use the symbol $A\subset\subset B$ if
$A$ is a closed subspace of a topological vector space $B$.
\begin{definition}
 Let $T\in\mathcal {L}(X,Y)$.
\begin{enumerate}
\item[{\rm (i)}]\ The {\it $n$th Kolmogorov number}\, of the
operator $T$ is  defined by
\begin{equation*}
d_n(T, X, Y)=\inf\{\|Q_N^YT\|:\,N\subset\subset Y,\,{\rm dim}
(N)<n\},
\end{equation*}
also written by $d_n(T)$ if no confusion is possible. Here, $Q_N^Y$
stands for the natural surjection of\,\,\,$Y$ onto the quotient
space $Y/N$.

\item[{\rm (ii)}]\ The {\it $n$th Gelfand number}\, of the operator
$T$ is  defined by
\begin{equation*}
c_n(T, X, Y)=\inf\{\|TJ_M^X\|:\,M\subset\subset X,\,{\rm codim}
(M)<n\},
\end{equation*}
also written by $c_n(T)$ if no confusion is possible. Here, $J_M^X$
stands for the natural injection of\,\,\,$M$ into $X$.
\end{enumerate}
\end{definition}

It is well-known that the operator $T$ is compact if and only if
$\lim_n d_n(T)=0$ or equivalently $\lim_n c_n(T)=0$, see
\cite{Pin85}.

The Kolmogorov and Gelfand numbers are dual to each other in the
following sense, cf. \cite{Pie78, Pin85}: If $X$ and $Y$ are
quasi-Banach spaces, then
\begin{equation}\label{dualc*d}
c_n(T^\ast)=d_n(T)
\end{equation}
for all compact operators $T\in\mathcal{L}(X, Y)$ and
\begin{equation}\label{duald*c}
d_n(T^\ast)=c_n(T)
\end{equation}
for all $T\in\mathcal{L}(X, Y).$\

For later proofs, we shall recall some basic facts about
approximation numbers. We define the $n$th approximation number of
$T$ by
\begin{equation}\label{an}
a_n(T)=\inf\{\|T - L\|:~ L\in \mathcal{L}(X,Y), ~\,{\rm rank} (L) <
n\},\quad n\in \mathbb{N},
\end{equation}
where ${\rm rank} (L)$ denotes the dimension of $L(X)$. We refer to
\cite{ET96,Pie78,Pin85} for detailed discussions of this concept and
further references.

Both, Gelfand and Kolmogorov numbers, are subadditive and
multiplicative s-numbers, as well as approximation numbers. One may
consult Pietsch \cite{Pie87}(Sections 2.4, 2.5), for the proof in
the Banach space case. Further, the generalization to $p$-Banach
spaces follows obviously. Let $s_n$ denote any of these three
quantities\ $a_n,\, d_n$ or $c_n$, and let $Y$ be a $p$-Banach
space,\ $0<p\le 1$. More precisely, we collect several common
properties of them as follows, \vspace{-0.2cm}
\begin{enumerate}
\item[]{\rm\bf(PS1)}\ (monotonicity)\,
$\|T\|=s_1(T)\ge s_2(T)\ge\cdots\ge 0$ for all $T\in\mathcal{L}(X,
Y)$,\vspace{-0.2cm}

\item[]{\rm\bf(PS2)}\ (subadditivity)\, $s_{m+n-1}^p(S+T)\leq
s_m^p(S)+s_n^p(T)$\, for all $m, n\in\mathbb{N},\,\,S,
T\in\mathcal{L}(X, Y)$,\vspace{-0.2cm}

\item[]{\rm\bf(PS3)}\ (multiplicativity)\, $s_{m+n-1}(ST)\leq
s_m(S)s_n(T)$\, for all $T\in\mathcal{L}(X, Y)$, $S\in\mathcal{L}(Y,
Z)$

\quad\quad and $m, n\in\mathbb{N},$\, cf. \cite[p. 155]{Pie78},
where $Z$ denotes a quasi-Banach space,\vspace{-0.2cm}

\item[]{\rm\bf(PS4)}\ (rank property)\, ${\rm rank}(T)<n$ if and only if
$s_n(T)=0$, where $T\in\mathcal{L}(X, Y)$.\vspace{-0.2cm}
\end{enumerate}

Moreover, there exist the following relationships: \beq\label{acd}
 c_n(T)\le a_n(T),~~~~ d_n(T)\le a_n(T),~~~ n \in \mathbb{N}.
\enq

Following Pietsch \cite{Pie87}, we will work with ideal quasi-norms
of $s$-numbers. For $0<r<\infty$, we set the following operator
ideals
\begin{equation}\mathscr{L}_{r,\infty}^{(s)}:=\left\{T\in\mathcal{L}(X,
Y):\quad \sup\limits_{n\in\mathbb{N}}n^{1/r}s_n(T)<\infty\right\}.
\end{equation}
 Equipped with the quasi-norm
\begin{equation}\label{idealddef}
L_{r,\infty}^{(s)}(T):=\sup\limits_{n\in\mathbb{N}}n^{1/r}s_n(T),
\end{equation}
the set $\mathscr{L}_{r,\infty}^{(s)}$ becomes a quasi-Banach space.
For such quasi-Banach spaces there always exists a real number
$0<\rho\leq 1$ such that
\begin{equation}\label{idealdinq}
L_{r,\infty}^{(s)}\left(\sum\limits_jT_j\right)^\rho\leq C
\sum\limits_jL_{r,\infty}^{(s)}(T_j)^\rho
\end{equation}
holds for any sequence of operators
$T_j\in\mathscr{L}_{r,\infty}^{(s)}$. Then we shall use the
quasi-norms $L_{r,\infty}^{(a)},~L_{r,\infty}^{(c)}$ and
$L_{r,\infty}^{(d)}$ for the approximation, Gelfand and Kolmogorov
numbers, respectively.

 Next we
define the weighted Besov and Triebel-Lizorkin spaces on
$\mathbb{R}^d$. For the definitions of the usual unweighted function
spaces of Besov and Triebel-Lizorkin type, and for more background,
we recommend \cite{ET96,Tr83,Tr06} as standard references. As usual,
$\mathcal{S}^\prime(\mathbb{R}^d)$ denotes the set of all tempered
distributions on the Euclidean $d$-space $\mathbb{R}^d$. Throughout
this paper we are interested in the function spaces with polynomial
weights given by (\ref{w_a}).
\begin{definition}\label{BF}
Let $0< p, q\leq \infty,\,$ and $s\in\mathbb{R}$. Then we put
\begin{equation*}
B_{p,q}^{s}(\mathbb{R}^d, w_\alpha)=\left\{f\in \mathcal{S}^\prime
(\mathbb{R}^d)\,:\, \|f ~|~ B_{p,q}^{s}(\mathbb{R}^d,
w_\alpha)\|=\|fw_\alpha ~|~
B_{p,q}^{s}(\mathbb{R}^d)\|<\infty\right\},
\end{equation*}
\begin{equation*}
F_{p,q}^{s}(\mathbb{R}^d, w_\alpha)=\left\{f\in \mathcal{S}^\prime
(\mathbb{R}^d)\,:\, \|f ~|~ F_{p,q}^{s}(\mathbb{R}^d,
w_\alpha)\|=\|fw_\alpha ~|~
F_{p,q}^{s}(\mathbb{R}^d)\|<\infty\right\},
\end{equation*}
with $p<\infty$ for the Triebel-Lizorkin spaces.
\end{definition}
\begin{re}
If no ambiguity arises, then we can write $B_{p,q}^{s}(w_\alpha)$
and $F_{p,q}^{s}(w_\alpha)$ for brevity.
\end{re}
\begin{re}
There are different ways to introduce weighted function spaces; cf.,
eg., Edmunds and Triebel \cite{ET96}, or Schmeisser and Triebel
\cite{ST87}. One can also consult \cite{KLSS05,Tr06,Tr08} for
related remarks.
\end{re}

Now we recall the necessary and sufficient condition for compactness
of the embeddings under consideration, which was proved in
\cite{HT94}, cf. also \cite{ET96,KLSS06}.
\begin{prop}\label{compact}
Let $w_\alpha$\ be as in (\ref{w_a}) with $\alpha>0.$\ Then the
embedding $B_{p_1,q_1}^{s_1}(\mathbb{R}^d,w_\alpha)\hookrightarrow
B_{p_2,q_2}^{s_2}(\mathbb{R}^d)$ is compact if and only if
$\min(\alpha, \delta)>d \max(\frac 1{p_2}-\frac 1{p_1}, 0)$.
\end{prop}

The same assertion holds for $F_{p,q}^s$-spaces with the restriction
$ p_1, p_2 <\infty$. We are now ready to formulate our main results.

\begin{theorem}\label{T1}
Let $\alpha>0,\ \ \delta\neq\alpha,\ \ \theta =
\frac{1/{p_1}-1/{p_2}}{1/2-1/{p_2}}$\  and \ $\frac
1{\tilde{p}}=\frac\mu d+\frac 1{p_1}$,\ where\
$\mu=\min(\alpha,\delta)$. Further, suppose $0< p_1\le
p_2\le\infty$\,~or~\,$\tilde{p}<p_2<p_1\le\infty$.

Denote by $d_n$ the $n$th Kolmogorov number of the embedding
(\ref{BB}). Then $d_{n}\sim n^{-\varkappa},$\ where \beg{enumerate}
\item[$(i)$]\ $\varkappa=\frac\mu d$\, if\, $0< p_1\le
p_2\le 2$\,\,or\,\,$2<p_1 = p_2\le \infty,$\vspace{-0.1cm}
\item[$(ii)$]\ $\varkappa =\frac \mu d+\frac 1{p_1}-\frac
1{p_2}$\, if\, $\tilde{p}<p_2<p_1\leq\infty,$\vspace{-0.1cm}
\item[$(iii)$]\ $\varkappa=\frac \mu d +\frac 12-\frac 1{p_2}$\,
 if\, $0< p_1 < 2 < p_2\le \infty$\,and \,$\mu>\frac d{p_2},$\vspace{-0.1cm}
\item[$(iv)$]\ $\varkappa=\frac \mu d\cdot\frac{p_2}2$\,
 if\, $0< p_1 < 2 < p_2< \infty$\,and \,$\mu<\frac d{p_2},$\vspace{-0.1cm}
\item[$(v)$]\ $\varkappa=\frac \mu d +\frac 1{p_1}-\frac 1{p_2}$\,
 if\, $2 \le p_1 < p_2 \le \infty$\,and \,$\mu>\frac d{p_2}\theta,$\vspace{-0.1cm}
\item[$(vi)$]\ $\varkappa=\frac \mu d\cdot\frac{p_2}2$\,
 if\, $2 \le p_1 < p_2 < \infty$\,and \,$\mu<\frac d{p_2}\theta.$
 \end{enumerate}
\end{theorem}

\begin{re}
We shift the proof of the above theorem to Subsection \ref{PrT1}.
And we wish to mention that both points, (iv) and (vi), vanish if\
$p_2=\infty.$\
\end{re}

\begin{re}
As is pointed in \cite{ZF10} (Remark 2.6),
Similar conclusions on the estimation of Kolmogorov numbers
of Sobolev embeddings on bounded domains could be made for Corollary 19
in \cite{Sk05}. The counterexample to our new part
(iv) could be also made for the limiting case $\delta = \frac d{p_2}$ according to \cite{SV09}.
\end{re}

For $0<p\le\infty,$ we set

$p^\prime=
\begin{cases}
\frac p{p-1}\quad &{\rm if}\ 1<p<\infty,\\
1 &{\rm if}\ p=\infty,\\
\infty &{\rm if}\ 0<p\le 1.
\end{cases}
$
\begin{theorem}\label{T2}
Let $\alpha>0,\ \ \delta\neq\alpha,\ \ \theta_1 =
\frac{1/{p_2^\prime}-1/{p_1^\prime}}{1/2-1/{p_1^\prime}}$\ and \
$\frac 1{\tilde{p}}=\frac\mu d+\frac 1{p_1}$,\ where\
$\mu=\min(\alpha,\delta)$.
 Further, suppose $0< p_1\le
p_2\le\infty$\,~or~\,$\tilde{p}<p_2<p_1\le\infty$.

 Denote by $c_n$ the $n$th Gelfand number
of the embedding (\ref{BB}). Then\ $c_{n}\sim n^{-\varkappa},$\
where
 \beg{enumerate}
\item[$(i)$]\ $\varkappa=\frac\mu d$\, if\, $2\le p_1\le p_2\le
\infty$\,\,or\,\,$0< p_1 = p_2< 2,$\vspace{-0.1cm}
\item[$(ii)$]\ $\varkappa =\frac \mu d+\frac 1{p_1}-\frac
1{p_2}$\, if\, $\tilde{p}<p_2<p_1\leq\infty,$\vspace{-0.1cm}
\item[$(iii)$]\ $\varkappa=\frac \mu d +\frac 1{p_1}-\frac 12$\,
 if\, $0< p_1 < 2 < p_2\le\infty$\,and \,$\mu>\frac d{p_1^\prime}$,\vspace{-0.1cm}
\item[$(iv)$]\ $\varkappa=\frac \mu d\cdot\frac{p_1^\prime}2$\,
 if\, $1< p_1 < 2 < p_2\le\infty$\,and \,$\mu<\frac d{p_1^\prime}$,\vspace{-0.1cm}
\item[$(v)$]\ $\varkappa=\frac \mu d +\frac 1{p_1}-\frac 1{p_2}$\,
 if\, $0 < p_1 < p_2 \le 2$\,and \,$\mu>\frac d{p_1^\prime}\theta_1$,\vspace{-0.1cm}
\item[$(vi)$]\ $\varkappa=\frac \mu d\cdot\frac{p_1^\prime}2$\,
 if\, $1< p_1 < p_2 \le 2$\,and \,$\mu<\frac d{p_1^\prime}\theta_1$.
\end{enumerate}
\end{theorem}
\begin{re}
We shift the proof of this assertion to Subsection \ref{PrT2}. Note
that both points, (iv) and (vi), vanish if\ $0<p_1\le 1.$\
\end{re}

\begin{re}
As well as in the Banach space case, we observe that
\begin{equation*}
d_n({\rm id}, ~B_{p_1,q_1}^{s_1}(\mathbb{R}^d, v_1),~
B_{p_2,q_2}^{s_2}(\mathbb{R}^d, v_2)) ~ \sim ~ d_n({\rm id},
~B_{p_1,q_1}^{s_1}(\mathbb{R}^d, v_1/v_2),~
B_{p_2,q_2}^{s_2}(\mathbb{R}^d)),
\end{equation*}
in the quasi-Banach case, where $v_1, v_2$ are admissible weight
functions. Moreover, the same formula holds for the Gelfand numbers.
Therefore, without loss of generality we can assume that the target
space is an unweighted space.
\end{re}

\beg{corollary}\label{BF}
 Theorem \ref{T1} and Theorem \ref{T2} remain valid
if instead of the embedding (\ref{BB}) we have any one of the
following embeddings:
\begin{equation*}
\begin{aligned}
F_{p_1,q_1}^{s_1}(\mathbb{R}^d, w_\alpha)\hookrightarrow
B_{p_2,q_2}^{s_2}(\mathbb{R}^d) ~  ~  ~  & {\rm if} ~ p_1<\infty,
\\
B_{p_1,q_1}^{s_1}(\mathbb{R}^d, w_\alpha)\hookrightarrow
F_{p_2,q_2}^{s_2}(\mathbb{R}^d) ~  ~  ~  & {\rm if} ~ p_2<\infty,
\\
F_{p_1,q_1}^{s_1}(\mathbb{R}^d, w_\alpha)\hookrightarrow
F_{p_2,q_2}^{s_2}(\mathbb{R}^d) ~  ~  ~  & {\rm if} ~
p_1,p_2<\infty.
\end{aligned}
\end{equation*}
\end{corollary}
\begin{re}
 We shift the short proof of
this corollary to Subsection \ref{PrBF}.
\end{re}
\begin{re}
As is noted in \cite{ZF10} (Remark 2.10),
for the limiting case $\delta=\alpha$, the exact order of related n-widths may possibly depend
on $q_1$ and $q_2$.  Some ideas
from \cite{HT05,KLSS05} may be helpful to further research in this situation.
\end{re}

\section{Preliminaries}
\subsection{Discretization of function spaces}
We use the discrete wavelet transform in order to
 transfer these problems from function spaces to
the corresponding sequence spaces, and then solve the task for the
sequence spaces. Afterwards, the results are transferred back to
function spaces. The crucial point in this discretization technique
is that the asymptotic order of the estimates is preserved.

\begin{prop}\label{Besov_des}
Let $s\in \mathbb{R}$ and $0< p,q\leq\infty.$ Assume $$r>\max(s,
\frac{2d}p+\frac d2-s).$$ Then for every weight $w_\alpha$ there
exists an orthonormal basis of compactly supported wavelet functions
$\{\varphi_{j,k}\}_{j,k}\cup\{\psi_{i,j,k}\}_{i,j,k},\
j\in\mathbb{N}_0,\ k\in \mathbb{Z}^d$ and $i=1, \ldots, 2^d-1$, such
that a distribution $f\in \mathcal {S}^\prime(\mathbb{R}^d)$ belongs
to $B_{p,q}^{s}(w_\alpha)$ if and only if
\begin{equation}\label{Besov_ell}
\begin{split}
\|f|B_{p,q}^{s}(w_\alpha)\|^\clubsuit =&\Big(
\sum\limits_{k\in\mathbb{Z}^d}|\langle f,\varphi_{0,k}\rangle
w_\alpha(k)|^p \Big)^{1/p}
\\
&+\sum\limits_{i=1}^{2^d-1}\Big\{ \sum\limits_{j=0}^\infty
2^{j\big(s+d(\frac 12-\frac 1p)\big)q}
\Big(\sum\limits_{k\in\mathbb{Z}^d}|\langle f,\psi_{i,j,k}\rangle
 w_\alpha(2^{-j}k)|^p\Big)^{q/p}\Big\}^{1/q}<\infty.
\end{split}
\end{equation}
Furthermore, $\|f|B_{p,q}^{s}(w_\alpha)\|^\clubsuit$ may be used as
an equivalent quasi-norm in $B_{p,q}^{s}(w_\alpha)$.
\end{prop}
\begin{re}
The proof of this proposition in its full generality may be found in
Haroske and Triebel \cite{HT05}. One can also consult \cite{KLSS05}
for historical remarks on various techniques of decompositions.
\end{re}
Let $0< p,q\leq\infty.$ Based on Proposition \ref{Besov_des} we will
work with the following weighted sequence spaces
\begin{equation}\label{ellal}
\begin{split}
\ell_q(2^{js}\ell_p(\alpha)):=\Bigg\{&
\lambda=(\lambda_{j,k})_{j,k}:~~\lambda_{j,k}\in\mathbb{C},\\
&\|\lambda|\ell_q(2^{js}\ell_p(\alpha))\|= \Big(
\sum\limits_{j=0}^\infty 2^{jsq}
\Big(\sum\limits_{k\in\mathbb{Z}^d}|\lambda_{j,k}\,w_{j,k}|^p
\Big)^{q/p}\Big)^{1/q}<\infty \Bigg\},
\end{split}
\end{equation}
(usual modification if $p=\infty$ and/or $q=\infty$), where
$w_{j,k}=w_\alpha(2^{-j}k)$. If $s=0$ we will write
$\ell_q(\ell_p(\alpha))$.

\subsection{Kolmogorov numbers of the Euclidean ball}\label{kns}

To begin with, we shall make preparations for the estimates of
Kolmogorov numbers of related function space embeddings in the
quasi-Banach setting with\ $0<p_1< 1$\ or\ $0<p_2< 1,$\ and for
several cases left over in the Banach setting with\ $p_2=\infty$.
The following result, Lemma \ref{knify}, is due to Garnaev and
Gluskin \cite{GG84}, Kashin \cite{Ka77} and Vyb\'iral \cite{Vy08}.

\beg{lemma}\label{knify} \ Let $N\in\mathbb{N}$ and $ n\le
N$.\vspace{-0.2cm} \beg{enumerate}
\item[$(i)$]\ If~ $1\le p < 2$ and $n\le\frac{N}{4}$\ then
$$n^{-1/2}\preceq  d_{n}\left({\rm id}, \ell_{p}^N, \ell_{\infty}^N\right)\preceq
 n^{-1/2}\Big(\log\big(\frac{eN}{n}\big)\Big)^{3/2}.$$\vspace{-0.6cm}
\item[$(ii)$]\ If~ $2 \le p < \infty$\, then
$$\frac 14\min\Big\{1,\Big(c_1\frac{\log(1+\frac N{n-1})}{n-1}\Big)^{1/p}\Big\}\le
 d_{n}\left({\rm id}, \ell_{p}^N, \ell_{\infty}^N\right)\le
\min\Big\{1,\Big(c_2\frac{\log(1+\frac
N{n-1})}{n-1}\Big)^{1/p}\Big\}$$ are valid for certain absolute
constants $c_1>0$\ and\  $c_2>0$.
\item[$(iii)$]\ If~ $0 < p_1<1$\ and\ $p_1 <p_2 \le \infty$\, then
$$d_{n}\big({\rm id}, \ell_{p_1}^N,
\ell_{p_2}^N\big)=d_{n}\big({\rm id}, \ell_{\min(1,p_2)}^N,
\ell_{p_2}^N\big).$$
\end{enumerate}
\end{lemma}
The following lemma is a simple corollary of Lemma \ref{knify}. And
the proof mimics that of Lemma 10 in \cite{Sk05}.
\beg{lemma}\label{knifyupp12}
 \ Let $1\le p_1 < 2$\
and\ $N=1, 2, 3, \ldots.$\ Then there is a positive constant $C$
independent of $N$ and $n$ such that \beq\label{kn12}
 d_{n}\left({\rm id}, \ell_{p_1}^N,
\ell_{\infty}^N\right)\le C\ \beg{cases}
 n^{-1/2}\Big(\log\big(\frac{4eN}{n}\big)\Big)^{3/2}\, ~ &{\rm if}\ 0<n\le
 N,\\0 &{\rm if}\ n>N.
\end{cases}
\enq
\end{lemma}

We wish to mention that, in contrast to Lemma \ref{gn021ify}, the
estimate
$$d_{n}\left({\rm id}, \ell_{p_1}^N, \ell_{p_2}^N\right)
=(N-n+1)^{\frac 1{p_2}-\frac 1{p_1}},\ \ 1\le n\le N\le\infty,$$ is
not valid for Kolmogorov numers if $ 0<p_2\le p_1\le \infty$\ and\
$p_2<1$. The following estimate from below was proved by Vyb\'iral
\cite{Vy08}.
 \beg{lemma}\label{kn021ify}
If~ $0<p_2\le p_1\le\infty,$ then there is a constant $c,\ 0<c\le
1,$\ such that
$$
d_{[cn]+1}\big({\rm id}, \ell_{p_1}^{2n},
\ell_{p_2}^{2n}\big)\succeq n^{1/{p_2}-1/{p_1}},\quad n\in
\mathbb{N},
$$
where $[cn]$ denotes the upper integer part of $cn$.
\end{lemma}

\subsection{Gelfand numbers of the Euclidean ball}\label{cns}
In this subsection we collect some known results on $c_n({\rm id},
\ell_{p_1}^N, \ell_{p_2}^N)$ for later use, cf.
\cite{FPRU,Ka77,LGM96,Pie78, Vy08}.

The following result is due to Foucart et al. \cite{FPRU}. Note that
the $n$-th Gelfand number is identical to the $(n-1)$-th Gelfand
width of $T$ defined in \cite{FPRU}, see also Pinkus \cite{Pin85}.
Here we recall it in our pattern.

 \beg{lemma}\label{gnupp}
 $\quad$ Let $1\le n\le N<\infty$.
\beg{enumerate}
\item[$(i)$]\ If~ $0< p_1\le 1$ and $2 < p_2\le\infty$\ then
there exist constants $C_1,\ C_2>0$\ depending only on\ $p_1$ and\
$p_2$\ such that
$$
C_1\min\Big\{1,\frac{\ln\big(\frac
 N{n-1}\big)+1}{n-1}\Big\}^{1/{p_1}-1/{p_2}} \le
 c_{n}\left({\rm id}, \ell_{p_1}^N, \ell_{p_2}^N\right)\le
C_2
 \min\Big\{1,\frac{\ln\big(\frac
 N{n-1}\big)+1}{n-1}\Big\}^{1/{p_1}-1/2}.$$

\item[$(ii)$]\ If~ $0< p_1\le 1$ and $p_1 < p_2\le 2$\ then
there exist constants $C_1,\ C_2>0$\ depending only on\ $p_1$ and\
$p_2$\ such that
$$
C_1\min\Big\{1,\frac{\ln\big(\frac
 N{n-1}\big)+1}{n-1}\Big\}^{1/{p_1}-1/{p_2}} \le c_{n}\left({\rm id},
\ell_{p_1}^N, \ell_{p_2}^N\right)\le
C_2\min\Big\{1,\frac{\ln\big(\frac
 N{n-1}\big)+1}{n-1}\Big\}^{1/{p_1}-1/{p_2}}.$$

\end{enumerate}
\end{lemma}

\begin{re}For the upper bounds considered above, there is another result given by
Vyb\'iral \cite{Vy08}, cf. Lemma 4.11, with a slight difference
between them on the log-factors. But they are equivalent to each
other for our upper estimates in Theorem \ref{T2}.
\end{re}

 \beg{lemma}\label{gnlow}
 $\quad$ Let $ n\in \mathbb{N}$.
\beg{enumerate}
\item[$(i)$]\ If~ $0< p_1\le 1$ and $2 \le p_2\le\infty$\ then
\beq\label{gnlow2}
 c_{n}\left({\rm id}, \ell_{p_1}^{2n},
\ell_{p_2}^{2n}\right)\succeq
 n^{1/2-1/{p_1}}.
 \enq

\item[$(ii)$]\ If~ $0< p_1\le 1$ and $p_1 < p_2\le 2$\ then
\beq\label{gnlow0} c_{n}\left({\rm id}, \ell_{p_1}^{2n},
\ell_{p_2}^{2n}\right)\succeq
 n^{1/{p_2}-1/{p_1}}.
 \enq

\end{enumerate}
\end{lemma}
The proof of the above lemma follows literally \cite[p. 567]{Vy08}, by
the multiplicativity of Gelfand numbers. In fact, The point (ii) in
Lemma \ref{gnupp} may also imply point (ii) of Lemma \ref{gnlow}.

 \beg{lemma}\label{gn021ify}
If~ $1\le n\le N<\infty$\ and $0<p_2\le p_1\le\infty,$ then
$$a_n\big({\rm id}, \ell_{p_1}^N, \ell_{p_2}^N\big)=
c_n\big({\rm id}, \ell_{p_1}^N, \ell_{p_2}^N\big)=
(N-n+1)^{1/{p_2}-1/{p_1}}.
$$
\end{lemma}
The proof of this lemma follows literally \cite{Pie78}, Section
11.11.4, see also \cite{Pin85}. Indeed the original proof is used
only to deal with the Banach setting. However, the same proof works
also in the quasi-Banach setting\ $0<p_2\le p_1\le\infty$.

\section{Proofs}

\subsection{Proof of Theorem \ref{T1}}\label{PrT1}

In \cite{ZF10} we were able to prove this theorem in the Banach
space case $1\le p,q \le \infty$, with the assumption that
$p_2<\infty$ holds if\ $p_1<p_2$. It is remarkable that the results
there do not depend on the fine indices\ $q_1$ and\ $q_2$.
 And the proof may
be directly generalized to the quasi-Banach setting\ $0<q_1,\,q_2\le
\infty,$\ with\ $1\le p_1,\,p_2\le \infty$. Afterwards, the
restrictions $q_1,q_2\ge1$ could be lifted.

Therefore, we may concentrate on the proof of
 \beg{enumerate}
\item[$(\clubsuit)$](i)\ if $0<p_1<1$\ and\ $p_1\le p_2\le 2$,
\item[$(\heartsuit)$](ii)\ if $0<\tilde{p}<p_2<p_1\leq\infty$\ and $0<p_2<1$,
\item[$(\spadesuit)$](iii) and (v)\ if $1\le p_1 < p_2=\infty$,
\item[$(\diamondsuit)$](iii) and (iv)\ if $0< p_1 <1$\ and\ $2 < p_2\le
\infty$.
\end{enumerate}
We shall give the proof of the estimates from above and below in
following four steps.

\emph{Step 1}: \emph{Proof of} $(\heartsuit)$.

To shorten notation define $1/p=1/{p_2}-1/{p_1}.$

We use the relation,\ $d_{n}\big({\rm id}, \ell_{p_1}^N,
\ell_{p_2}^N\big))\le a_{n}\big({\rm id}, \ell_{p_1}^N,
\ell_{p_2}^N\big)$, cf. (\ref{acd}).\ Note that Step 1 of the proof
of Proposition 15 in \cite{Sk05} may be directly generalized to the
quasi-Banach setting, where\ $0<p_2< p_1\le\infty$\ and\ $0<q_1,
q_2\le\infty.$\ So similar arguments give the estimates from above
as required.

For the estimates from below, we obtain by Lemma \ref{kn021ify} that
\beq\label{kn021}
 d_n\big({\rm id}, \ell_{p_1}^{N},
\ell_{p_2}^{N}\big)\succeq n^{1/p},
 \enq
for $n=\big[\frac c2\cdot N\big],$ \ where $c$ is the constant from
Lemma \ref{kn021ify}. We can deal with the estimates in a similar
manner as in Step 4 of the proof of Proposition 11 in \cite{Sk05},
by using (\ref{kn021}). Note that in order to guarantee the
compactness of the embeddings here we only need to consider two
cases, $ d/p<\alpha<\delta$\ or \ $ d/p<\delta<\alpha$, instead of
four cases.

\emph{Step 2}: \emph{Proof of} $(\clubsuit)$.

If $p_2\ge 1,$\ we obtain by Lemma \ref{knify} that
\beq\label{kn11+} d_{n}\big({\rm id}, \ell_{p_1}^N,
\ell_{p_2}^N\big)=d_{n}\big({\rm id}, \ell_1^N,
\ell_{p_2}^N\big).\enq

Then the estimates may be shifted immediately to the Banach case, $1= p_1\le p_2\le2$. Similar arguments give the sharp two-sided estimates.

If $p_2<1$ and $n\le N$, then
$$d_{n}\big({\rm id}, \ell_{p_1}^N,
\ell_{p_2}^N\big)=d_{n}\big({\rm id}, \ell_{p_2}^N,
\ell_{p_2}^N\big).$$

Thereby, the proof of the upper and lower estimates follows in the same way as in the first step. Two cases, $0<\alpha<\delta$\ or \ $0<\delta<\alpha$, are considered instead.

\emph{Step 3}: \emph{Proof of} $(\spadesuit)$.

For the estimates from above, we still adopt the operator ideal.
(\ref{kn12}) and Lemma \ref{knify} (ii) imply that

\beq\label{Lknify} L_{s,\infty}^{(d)}\left({\rm id}, \ell_{p_1}^N,
\ell_{\infty}^N\right) \le C\ \beg{cases}
 N^{1/s-1/2}\, ~ & {\rm if}\ 1\le p_1<2\  {\rm and}\ \frac 1s>\frac
 12,
 \\N^{1/s-1/p_1}\, ~ & {\rm if}\ 2\le p_1<\infty\  {\rm and}\ \frac 1s>\frac
 1{p_1}.
\end{cases}
 \enq
 Note that related computations above is similar to that of ideal quasi-norms of entropy
 numbers, cf. \cite{ET96,KLSS06}.
 Next we proceed as in the proof of Proposition 11 in \cite{Sk05}.
As to the definition of $P_{i,j},\ i,j\in\mathbb{N}_0$, we refer to
the counterpart there again. For any given $M\in\mathbb{N}_0$,\, we
also put
\begin{equation}\label{PQ}
P:=\sum\limits_{m=0}^M\sum\limits_{j+i=m}P_{j,i}\quad\quad{\rm
and}\quad\quad
Q:=\sum\limits_{m=M+1}^\infty\sum\limits_{j+i=m}P_{j,i}.
\end{equation}
Set $\beta=\max(2, p_1)$. For $L_{s,\infty}^{(d)}\left(P\right)$,
we choose $s$ such that
 $\frac 1s>\frac \mu d+\frac 1\beta$. And for
 $L_{s,\infty}^{(d)}\left(Q\right)$, we select $s$ satisfying
 $\frac 1\beta <\frac 1s<\frac \mu d+\frac 1\beta$.
 Once more the upper estimates are finished.

For the lower estimates, we obtain by Lemma \ref{knify} that, for
$n=\big[\frac N4\big]$, \beq\label{kn1-}
 d_n\big({\rm id}, \ell_{p_1}^{N},
\ell_\infty^{N}\big) \succeq n^{-1/\beta},\ \ {\rm where} \
\beta=\max(2, p_1).
 \enq
Then we only consider two cases, $ 0<\alpha<\delta$\ or \ $
0<\delta<\alpha$ instead. Again, we follow Step 4 of the proof of
Proposition 11 in \cite{Sk05}, now using (\ref{kn1-}) instead.

 \emph{Step 4}: \emph{Proof of} $(\diamondsuit)$.

If $p_2<\infty,$\ the estimates may be shifted immediately by
point (iii) of Lemma \ref{knify} to the Banach case, $1= p_1<2\le p_2<\infty,$
considered in \cite{ZF10}. Similar arguments give the sharp
two-sided estimates.

If $p_2=\infty,$\ then the point (iv) vanishes, and for
$(\diamondsuit)$ in point (iii) we follow trivially the third step by
Lemma \ref{knify} (iii).\qed

\subsection{Proof of Theorem \ref{T2}}\label{PrT2}

Arguments similar to the last proof lead us to lift the restrictions
$q_1,q_2\ge1$ in those results for Gelfand numbers obtained in
\cite{ZF10}. So we can concentrate on the proof of
 \beg{enumerate}
\item[$(\clubsuit)$](i)\ if $0<p_1=p_2<1,$
\item[$(\heartsuit)$](ii)\ if $0<\tilde{p}<p_2<p_1\leq\infty$\ and $0<p_2<1,$
\item[$(\spadesuit)$](iii)\ if $0< p_1 \le 1$\ and\ $2 < p_2\le
\infty,$
\item[$(\diamondsuit)$](v)\ if $0< p_1 \le 1$ \ and\ $p_1 < p_2\le 2.$
\end{enumerate}
We shall give the proof of the estimates from above and below in
following four steps.

\emph{Step 1}: \emph{Proof of} $(\clubsuit)$.

If $0<p_1=p_2<1$ and $n\le N$, then it holds obviously that
$$c_{n}\big({\rm id}, \ell_{p_1}^N,
\ell_{p_2}^N\big)=c_{n}\big({\rm id}, \ell_{p_2}^N,
\ell_{p_2}^N\big)=1.$$ The proof of $(\clubsuit)$ follows literally that of
Proposition 13 in \cite{Sk05}.

\emph{Step 2}: \emph{Proof of} $(\heartsuit)$.

By Lemma \ref{gn021ify}, the proof of $(\heartsuit)$ follows exactly
as in the proof for the Banach case given in \cite{ZF10}, see also
\cite{Sk05} (Proposition 15) for further details.

\emph{Step 3}: \emph{Proof of} $(\spadesuit)$.

We can deal with the proof of $(\spadesuit)$ in a way similar to that of
$(\spadesuit)$ in Theorem \ref{T1}.

For the upper estimates, we first obtain by Lemma \ref{gnupp} that
\beq\label{gnupp2+} L_{s,\infty}^{(c)}\left({\rm id}, \ell_{p_1}^N,
\ell_{p_2}^N\right) \le C
 N^{1/s-(1/{p_1}-1/2)}\quad {\rm if}\ \frac 1s> \frac 1{p_1}-\frac 12.
 \enq
Again, we adopt the notations $P$ and $Q$ from \cite{Sk05}. Choose $s$ such that $\frac 1s>\frac \mu d+\frac 1{p_1}-\frac 12$
for $L_{s,\infty}^{(c)}\left(P\right)$,\ and
 $\frac 1{p_1}-\frac 12<\frac 1s<\frac \mu d+\frac 1{p_1}-\frac 12$\ for\
 $L_{s,\infty}^{(c)}\left(Q\right),$ \ respectively.
 Once more the upper estimates are complete.

For the lower estimates, we use (\ref{gnlow2}) with\ $n=\big[\frac
N2\big]$.  Once more we follow the pattern of Step 4 in the proof of
Proposition 11 in \cite{Sk05}, now dealing with two cases,
$0<\alpha<\delta$\ or \ $ 0<\delta<\alpha$.

 \emph{Step 4}: \emph{Proof of} $(\diamondsuit)$.

The proof of $(\diamondsuit)$ follows literally as
in the third step, with
 Lemma \ref{gnupp} (i) replaced by Lemma
\ref{gnupp} (ii) for the upper bounds, and (\ref{gnlow2}) replaced
by (\ref{gnlow0}) for the lower bounds. \qed

\subsection{Proof of Corollary \ref{BF}}\label{PrBF}
We denote by
$A_{p_,q}^{s}(\mathbb{R}^d,w_\alpha)\,(A_{p_,q}^{s}(\mathbb{R}^d))$
either $B_{p_,q}^{s}(\mathbb{R}^d,w_\alpha)
(B_{p_,q}^{s}(\mathbb{R}^d))$ or
$F_{p_,q}^{s}(\mathbb{R}^d,w_\alpha) (F_{p_,q}^{s}(\mathbb{R}^d)),$
with the restraint that for the $F$-spaces $p<\infty$ holds. The
proof follows by the relations as below, see \cite[Section
2.2.3, p. 44]{ET96} and \cite{Tr83} for further details,

 \beq B_{p_,u}^{s}(\mathbb{R}^d)\hookrightarrow
F_{p_,q}^{s}(\mathbb{R}^d)\hookrightarrow
B_{p_,v}^{s}(\mathbb{R}^d), \enq where\ $0<u\le\min(p,q)$\ and\
$\max(p,q)\le v\le \infty.$

For the embeddings of these three types, we prove the upper bounds,
by virtue of the multiplicativity property of Kolmogorov and Gelfand
numbers, and the following embeddings
\begin{equation*}
A_{p_1,q_1}^{s_1}(\mathbb{R}^d,w_\alpha) \hookrightarrow
B_{p_1,\infty}^{s_1}(\mathbb{R}^d,w_\alpha) \hookrightarrow
B_{p_2,u_1}^{s_2}(\mathbb{R}^d) \hookrightarrow
A_{p_2,q_2}^{s_2}(\mathbb{R}^d),
\end{equation*}
where\ $u_1=\min(p_2, q_2).$

 For the estimate from below we can consider the following
embeddings
\begin{equation*}
B_{p_1,u_2}^{s_1}(\mathbb{R}^d,w_\alpha) \hookrightarrow
A_{p_1,q_1}^{s_1}(\mathbb{R}^d,w_\alpha) \hookrightarrow
A_{p_2,q_2}^{s_2}(\mathbb{R}^d) \hookrightarrow
B_{p_2,\infty}^{s_2}(\mathbb{R}^d),
\end{equation*}
where\ $u_2=\min(p_1, q_1).$ \qed

\section{Comparisons with approximation numbers}

In this closing section we wish to compare the approximation,
Gelfand and Kolmogorov numbers of Sobolev embeddings between
weighted function spaces of Besov and Triebel-Lizorkin type in the
non-limiting situation.

Let us first complement the known results for the approximation
numbers. Specifically, in \cite{Sk05,SV09} the exact estimates of
approximation number were established in almost all cases. However,
the problem was still open in case when $0< p_1 \le 1$ and
$p_2=\infty$. Here we are able to close the gaps in the non-limiting
situation.

\begin{lemma}\label{an1inf}
Let $0 < p\le 1$ and $N\in\mathbb{N}$. \vspace{-0.2cm}
\beg{enumerate}
\item[$(i)$] Let $0<\lambda<1$. Then there exists a constant
$C_\lambda>0$ depending only on $\lambda$ such that
\begin{equation}\label{an1infupp}
a_{n}\big({\rm id}, \ell_{p}^N, \ell_\infty^N\big)\le
\begin{cases}
1 & {\rm if}\ \ n\leq N^\lambda,\\
C_\lambda n^{-1/2}\quad & {\rm if}\ \ N^\lambda<n\le N,\\
0 & {\rm if}\ \ n>N. \end{cases}
\end{equation}
\vspace{-0.6cm}
\item[$(ii)$] There exists a constant
$C>0$ independent of $n$ such that for any  $n\in \mathbb{N}$
\begin{equation}\label{an1inflow}
a_{n}\big({\rm id}, \ell_{p}^{2n}, \ell_\infty^{2n}\big)\ge C
n^{-1/2}.
\end{equation}
\end{enumerate}
\end{lemma}
We refer to \cite{Vy08} for the proof.

\begin{theorem}\label{an}
Let $\alpha>0,\ \delta\neq\alpha,\ t=\min(p_1^\prime,p_2)$  and
$\frac 1{\tilde{p}}=\frac\mu d+\frac 1{p_1}$,\ \,\emph{where}\
$\mu=\min(\alpha,\delta)$.  Further, suppose $0< p_1\le
p_2\le\infty$\,~or~\,$\tilde{p}<p_2<p_1\le\infty$.

Denote by $a_n$ the $n$th approximation number of the Sobolev
embedding\vspace{-0.1cm}
\begin{equation}\label{aa}
A_{p_1,q_1}^{s_1}(\mathbb{R}^d,w_\alpha)\hookrightarrow
A_{p_2,q_2}^{s_2}(\mathbb{R}^d).
\end{equation}
 Then $a_{n}\sim n^{-\varkappa},$\ where
 \beg{enumerate}
\item[$(i)$]\ $\varkappa=\frac\mu d$\, if\, $0< p_1\le
p_2\le 2$\,\,or\,\,$2\le  p_1\le p_2\le \infty$,\vspace{-0.1cm}
\item[$(ii)$]\ $\varkappa =\frac \mu d+\frac 1{p_1}-\frac
1{p_2}$\, if\, $\tilde{p}<p_2<p_1\leq\infty$,\vspace{-0.1cm}
\item[$(iii)$]\ $\varkappa=\frac \mu d +\frac 12-\frac 1t$\,
 if\, $0< p_1 < 2 < p_2\le \infty$\,and \,$\mu>\frac dt$,\vspace{-0.1cm}
\item[$(iv)$]\ $\varkappa=\frac \mu d\cdot\frac t2$\,
 if\, $0< p_1 < 2 < p_2\le \infty$\,and \,$\mu<\frac dt$.
 \end{enumerate}
\end{theorem}
\begin{proof}
 We only sketch the proof in the case when $0< p_1 \le 1$ and $p_2=\infty$.
For the proofs in the other cases, one can consult \cite{Sk05}, cf.
also \cite{ET96} for the asymptotic estimates of the Euclidean ball
in the quasi-Banach case. The crucial point in the following step is the proper choice of $\lambda$.

\emph{Step 1} (Upper estimates).  $0<p_1\le1$ and $p_2=\infty$ imply
$t=\infty$. We select $0<\lambda<1$ such that
$\frac\lambda{2(1-\lambda)}<\frac{\min(\alpha,\delta)} d$. The
inequality $\lambda\cdot \frac 1s\le\frac 1s-\frac 12$ holds if and
only if $\frac 1s\ge\frac 1{2(1-\lambda)}$, where $0<\lambda<1$.
Then, we find by (\ref{an1infupp}) that for any $N\in\mathbb{N}$
\begin{equation}\label{idealseh}
L_{h,\infty}^{(a)}({\rm id}, \ell_{p_1}^N, \ell_{p_2}^N)\leq C\
N^{\frac\lambda{2(1-\lambda)}},\ \ \ {\rm if}\ \ \frac 1h=\frac
1{2(1-\lambda)},
\end{equation}
\begin{equation}\label{idealsh}
L_{s,\infty}^{(a)}({\rm id}, \ell_{p_1}^N, \ell_{p_2}^N)\leq C\
N^{(\frac 1s-\frac 12)},\ \ \ {\rm if}\ \ \frac 1s>\frac
1{2(1-\lambda)}.
\end{equation}
Next, our proof may mimic that of Proposition 11 in \cite{Sk05}. As
to the precise definitions of $P$ and $Q$, we refer to the
counterpart there again. For the estimation of $a_n(P)$, we choose
$s$ such that $\frac 1s>\frac 1{2(1-\lambda)}$\ and\ $d(\frac
1s-\frac 12)>\min(\alpha,\delta)$, and proceed by using
(\ref{idealsh}). For the estimation of $a_n(Q)$, we choose
$s=h=2(1-\lambda)$, and use (\ref{idealseh}) instead. Note that
$\lambda\cdot \frac 1h=\frac 1h-\frac 12<\frac{\min(\alpha,\delta)}
d$ for the above choice of $\lambda$.

\emph{Step 2} (Lower estimates). We only need to consider two cases,
$0<\delta\le\alpha$\
 or $0<\alpha\le\delta$. And we choose $\ell=\big[\frac N2\big]$ where $N$ is
taken in the same way as in point $(i)$ or $(ii)$ of Step 4 of
Proposition 11 in \cite{Sk05}, respectively, and use
(\ref{an1inflow}).
\end{proof}

\begin{re}
Note that in the above assertion point (iv) vanishes if\ $0< p_1 \le
1$ and $p_2=\infty$. Moreover, the two function spaces in the
embedding (\ref{aa}) may be of different types, i.e., one is the
Besov space, and the other is the Triebel-Lizorkin space.
\end{re}

 The comparison of these above theorems, shows that,
for Sobolev embeddings of weighted function spaces of Besov and
Triebel-Lizorkin type, with $p<\infty$ for the
$F$-spaces,\vspace{-0.1cm}
\begin{enumerate}
\item[$(i)$]\ $a_n\sim c_n$ if either

$(a)$\ $2\le p_1< p_2\le \infty$\ or,

$(b)$\ $\tilde{p}<p_2\le p_1\le \infty$\ or,

$(c)$\ $1\le p_1 < p_1^\prime\le p_2\le\infty$\ and \
$\min(\alpha,\delta)\neq\frac d{p_1^\prime}$; \vspace{-0.1cm}

\item[$(ii)$]\ $a_n\sim d_n$ if either

$(a)$\ $0< p_1< p_2\le 2$\ or,

$(b)$\ $\tilde{p}<p_2\le p_1\le \infty$\ or,

$(c)$\ $0< p_1 < 2 < p_2\le p_1^\prime\le\infty$\ and\
$\min(\alpha,\delta)\neq\frac d{p_2}$;\vspace{-0.1cm}

\item[$(iii)$]\ $c_n\sim d_n$ if either

$(a)$\ $\tilde{p}<p_2\le p_1\le \infty$\ or,

$(b)$\ $1\le p_1 < p_1^\prime= p_2\le\infty$\ and\
$\min(\alpha,\delta)\neq\frac d{p_2}$.
\end{enumerate}
\section*{Acknowledgments}
~~~~ The authors are extremely grateful to Ant\'onio M. Caetano,
Thomas K$\ddot{\rm u}$hn, Erich Novak, Leszek Skrzypczak and Jan Vyb\'iral for
their direction and help on this work.

This work was partially supported by the Natural Science Foundation
of China (Grant No. 10671019, No. 11171137, No. 61175046), Anhui
Provincial Natural Science Foundation (No. 090416230) and Youth Foundation of Anhui University (No. 33050069).

\end{document}